\documentclass[11pt, A4paper]{article}

\usepackage{amssymb,bbm,amsmath,mathrsfs}

\usepackage{graphicx}

\usepackage{amsthm}
\usepackage{enumitem}
\usepackage{dsfont}
\usepackage{color}

\usepackage[unicode]{hyperref}
\usepackage[unicode]{hyperref}
\hypersetup{
  colorlinks=true,    % Aktiviert farbige Links
  linkcolor=black,    % Farbe interner Links
  citecolor=black,    % Farbe von Zitierlinks
  filecolor=black,    % Farbe von Dateilinks
  urlcolor=black,     % Farbe von URL-Links
  pdfborder={0 0 0}   % Kein Rahmen um Links
}

\usepackage[utf8]{inputenc}
\usepackage[T1]{fontenc}
\usepackage{geometry}
\usepackage{todonotes}
\usepackage{lmodern}
\usepackage{anyfontsize}
\usepackage{cleveref}
\usepackage{comment}
\usepackage{makecell}

\usepackage{tikz-cd}
\usepackage{algorithm}
\usepackage{algpseudocode}

\usepackage{geometry}
\usepackage{marginnote}
\newcommand{\MBtext}[1]{{{\small\color{purple}{}}}}

\newcommand{\LRlong}[1]{{{\small\color{cyan}{}}}}

\newcommand{\LRtext}[1]{{{\small\color{cyan}{}}}}

\theoremstyle{plain}
\newtheorem{theorem}{Theorem}[section]
\newtheorem{proposition}[theorem]{Proposition}
\newtheorem{lemma}[theorem]{Lemma}
\newtheorem{corollary}[theorem]{Corollary}

\theoremstyle{definition}

\newtheorem{remark}[theorem]{Remark}

\newtheorem{example}[theorem]{Example}

\newcommand{\R}{\mathbb{R}}
\newcommand{\RR}{{\mathbb{R}_{++}}}

\newcommand{\E}{\mathbb{E}}

\newcommand{\M}{\mathscr{M}}

\newcommand{\cF}{\mathcal{F}}

\newcommand{\cP}{\mathcal{P}}

\newcommand{\scM}{\mathscr{M}}

\newcommand{\proj}{{\rm proj}}
\newcommand{\law}{{\rm law}}

\newcommand{\AW}{\mathcal{AW}}

\newcommand{\cplbc}{\Pi_{\rm bc}}

\newcommand{\SBM}{{\rm SBM}}
\newcommand{\gSBM}{{\rm gSBM}}

\newcommand{\F}{\mathcal F}

\newcommand{\G}{\mathcal G}

\renewcommand{\P}{\mathcal P}
\newcommand{\PP}{\mathbb P}
\newcommand{\QQ}{\mathbb Q}
\newcommand{\D}{\mathcal D}

\newcommand{\MCov}{\textnormal{MCov}}

\newcommand{\argmin}{\textnormal{argmin}}

\author{M.\ Beiglböck, G.\ Pammer, and L.\ Riess}

\begin{document}

\title{Change of numeraire for weak martingale transport}

\maketitle

%\author{M.\ Beiglböck, G.\ Pammer, and L.\ Riess}\thanks{Financial support through FWF-projects Y0782 and P35197 is gratefully acknowledged.}

\begin{abstract}
Change of numeraire is a classical tool in mathematical finance. Campi--Laachir--Martini \cite{CaLaMa14} established its applicability to martingale optimal transport. We note that the results of \cite{CaLaMa14} extend to the case of weak martingale transport. We apply this to shadow couplings (in the sense of \cite{BeJu21}), continuous time martingale transport problems in the framework of Huesmann--Trevisan \cite{HuTr17} and in particular to establish the correspondence between stretched Brownian motion with its geometric counterpart.\footnote{We emphasize that we learned about the geometric stretched Brownian motion \gSBM\ (defined in PDE terms) in a presentation of Loeper \cite{Lo23} before our work on this topic started. We noticed that a change of numeraire transformation in the spirit of \cite{CaLaMa14} allows for an alternative viewpoint in the weak optimal transport framework. We make our work public following the publication of Backhoff--Loeper--Obloj's work \cite{BaLoOb24} on {\tt arxiv.org}. The article \cite{BaLoOb24} derives \gSBM\  using  PDE techniques as well as through an independent probabilistic approach which is close to the one we give in the present article.}

\medskip

\noindent{\it }
\end{abstract}

\section{Overview}

While classical transport theory is concerned with the set $\Pi(\mu, \nu)$ of \emph{couplings} or \emph{transport plans} of probability measures $\mu,\nu$, the martingale variant restricts the problem to the set $\scM_{\{0,1\}}(\mu,\nu)$ of \emph{martingale couplings}, i.e.\ laws of two step martingales $(X_t)_{t\in\{0,1\}}$ satisfying $X_0 \sim \mu, X_1\sim \nu$. The investigation of such martingale transport plans is primarily motivated by applications in mathematical finance. In this case the interest lies in probabilities $\mu, \nu$ which are concentrated on the positive half line $\R_{++}= (0,\infty)$, have finite first moments and are in convex order. We make this a standing assumption throughout. 

In this article we revisit the change of numeraire technique which was first used in the context of martingale optimal transport by Campi--Laachir--Martini \cite{CaLaMa14}. This \emph{CN-transformation} (for brevity) converts a martingal transport plan $\pi$ between $\mu, \nu$ into a martingale transport plan `$S (\pi) $' between altered marginals (`$S(\mu), S(\nu)$') changing cost functionals on the way. In particular this operation allows to transfer known solutions of martingale transport problems to solutions of altered problems. 

\medskip

In Section \ref{sec:CLM_trafo_new_sec} we summarize the basic properties of the CN-transformation in the case of \emph{weak martingale optimal transport} (wMOT). Weak versions of transport problems were introduced by Gozlan--Roberto--Samson--Tetali \cite{GoRoSaTe15} and allow to consider non-linear cost functionals, see e.g.\ the survey \cite{BaPa20} and \cite{BeJoMaPa21b} for an overview in the martingale context. 

We then collect different applications:

In Section \ref{sec:AppgSBM} we apply the CN-transformation to stretched Brownian motions (or `Bass martingales'), \cite{Lo18, BaBeHuKa20, GuLoWa19,  BaBeScTs23}. Stretched Brownian motion (\SBM) is the process which has maximal correlation with Brownian motion, subject to having $\mu, \nu$ as prescribed initial and terminal marginals.
From a mathematical finance perspective it is natural to consider a geometric stretched Brownian motion (\gSBM) where the log returns maximize the correlation with Brownian motion subject to the marginal constraints. We show that \gSBM{} is the CN-transform of \SBM{} and provide numerical simulations based on the known algorithm for \SBM{} (\cite{CoHe21, AcMaPa23, JoLoOb23}).

In Section 3 we consider the adapted Wasserstein distance $\AW$ which provides a metric between continuous martingales which takes the flow of information into account. It is known from \cite{BaBeHuKa20} that stretched Brownian motion is the $\AW$-metric projection of Brownian motion onto the set of martingales with prescribed initial and terminal marginal. We introduce the geometric counterpart $g\AW$ of the adapted Wasserstein distance and establish that $\gSBM$ is a $g\AW$-metric projection.

In Section \ref{Sec:AppHT} we revisit continuous-time martingale transport problems in the spirit of Huesmann--Trevisan \cite{HuTr17} and Guo--Loeper--Wang \cite{GuLoWa19}.

In Section \ref{sec:AppShadow} we consider the class of shadow couplings of martingale transport plans, see \cite{BeJu21}. We find that the CN-transformation of a shadow coupling is again a shadow coupling but with `inverted' source. In particular this implies that left-monotone martingale couplings are transformed into right-monotone martingale couplings (recovering a result of \cite{CaLaMa14}) and that `sunset-couplings' are invariant under the CN-transformation.

\subsection{Martingale transport and the CN-transformation}\label{sec:CLM_trafo_new_sec}

We consider a set of time indices $I\subset [0,\infty)$ which is compact and denote its minimal element by $o$ and its maximal element by $T$. Examples which we have in mind are $I= \{0,1\}, I=[0,1]$ and $ I= \{t\}$ where $t\in [0,\infty)$. Throughout, we will denote for a probability $\eta$ on $\R$ and a measurable function $f:\R\to\R$, the push-forward, by $f_\#\eta$. Given a function $g$ on $\R$ we will write $\eta(g):=\int g\, d\eta$ whenever the integral exists. Additionally, we write $b(\eta)$ for the barycenter of the probability $\eta$.

Let $\mu, \nu\in \cP_1(\RR)$, that is, $\mu$ and $\nu$ are probabilities on $\RR$ with finite first moments.
We write $\M_I$ for the set of all laws of (cadlag) martingales $(X_t)_{t\in I}$ defined on some probability space $(\Omega, \F, \PP)$, and denote by $\M_I(\mu, \nu)$ the subset of martingale transport plans additionally satisfying $X_o\sim \mu, X_T\sim \nu$. A \emph{cost function} is a measurable, $[0,\infty]$-valued function $c$ defined on the \emph{path space} $\mathcal D_I$ of cadlag $\R_{++}$-valued functions on $I$. The corresponding martingale transport problem consists in 
\begin{align*}
    \inf\Big\{ \int c \, d\pi: \pi \in \M_I(\mu, \nu)\Big\}=\inf\Big\{ \E_\PP [c (X)]: X \text{  mart., $\law_\PP(X_o)= \mu, \law_\PP(X_T)= \nu  $}\Big\}. 
\end{align*}
Given a measurable lower bounded `generalized' cost function $C:\mathcal P_1(\RR) \to \R \cup \{ \infty \}$, the weak martingale transport problem is
\begin{multline*}
    \inf\Big\{ \int C(\pi_x) \, \mu(dx): \pi \in \M_{\{0,1\}}(\mu, \nu)\Big\} \\= \inf\Big\{ \E_\PP[C(\law_\PP(X_1|X_0))]: X \text{ mart., $\law_\PP(X_0)= \mu, \law_\PP(X_1)= \nu  $}\Big\}, 
\end{multline*}
where we use $(\pi_x)_x$ to denote the disintegration of $\pi$ with respect to the first marginal. 
Usually cost functions are assumed to be lower semi-continuous and $C$ is assumed to be convex which guarantees the existence of optimizers as well as weak duality, but this is not important at the current stage. 

\medskip

Next we discuss the classical change of numeraire transformation: Recall that \emph{Bayes' rule} asserts that for integrable random variables $D,Z$ on $(\Omega, \F, \PP)$, where $D>0$ and $\G\subseteq \F$ and $\hat \PP:= \frac{D}{\E_\PP[D]} \PP$ 
\[ \E_{\hat \PP}[Z|G] =  \frac {\E_\PP[ZD|\G]}{\E_\PP[D|G]}. \]
Given a path $x=(x_t)_{t\in I}$ we write $1/x$ for the path $(1/x_t)_{t\in I}$.
For a $\PP$-martingale $X=(X_t)_{t\in I}$, the process
 $Y:=1/X$ is a $\hat \PP$-martingale where $\hat\PP:= \frac{X_T}{\E_\PP[X_T]} \PP$. This follows directly from Bayes rule since for $t\in I$
$$\E_{\hat \PP} [Y_T|Y_t] = \frac {\E_\PP[Y_T X_T|\F_t]}{\E_\PP[X_T|\F_t]} = \frac{1}{X_t}=Y_t.$$
Importantly, $\law_{\hat\PP}(Y)$ depends only on $\law_\PP(X)$ since for $\pi:= \law_\PP(X)$ and measurable $B\subseteq \mathcal D(I)$ 
\begin{align*}
    \law_{\hat \PP}(Y)(B)= \E_\PP \Big[ \frac{X_TI_{\{1/X\in B\}}}{\E_\PP[X_T] } \Big] = \frac{\int x_T \delta_{1/x}(B)\,  \pi(dx) }{ \int x_T \, \pi(dx) }.
\end{align*}
We will call this transformation
\begin{align} \label{eq:def.CLM-pi_unified}
    S:\M_I \to \M_I, \quad S(\pi)(B) := \frac{\int x_T \delta_{1/x} (B)\, \pi(dx) }{ \int x_T \, \pi(dx)},\quad \text{for measurable $B\subseteq \D_I$}  
\end{align}
the \emph{CN}-transformation. 
We identify the sets $\M_{\{t\}}$ (where $t\in \R$) and the set $\cP_1(\R_{++})$ of probabilities with finite first moments such that $S(\pi) $ is also defined for $\pi\in \cP_1(\R_{++})$.

In the following two theorems we derive the basic properties of $S$.
In the statement as well as in the proof we will use the representation $S(\pi)= \law_{\hat\PP}(Y)$, where $\pi, X,Y, \PP, \hat \PP$ are as above. 
For $J\subseteq I$ we write $\proj^{\D_J} $ for the projection $\proj^{\D_J}:\D_I\to \D_J,(x_t)_{t\in I}\mapsto (x_t)_{t\in J}$.

\begin{theorem}\label{thm:first_properties_CLM}
    The map $S$ has the following properties:
    \begin{enumerate}[label = (\roman*)]
        \item $S$ is an involution, i.e.\
        \begin{align}\label{eq:Involution} \tag{P1}
            S(S(\pi))= \pi\quad  \text{for $\pi\in \M_I$}. 
        \end{align}
        \item If $J\subseteq I$ is compact, in particular if $J=\{t\}$ for some $t\in I$, we have 
            \begin{align}\label{eq:ProjOk} 
                \proj_\#^{\D_J}(S(\pi)) = S(\proj_\#^{\D_J}(\pi)) \quad  \text{for $\pi\in \M_I$}.  \tag{P2}
            \end{align}
        \item For $\mu, \nu\in \P_1(\R_{++})$
            \begin{align}\label{eq:PlanBij} 
                S:\scM_I(\mu, \nu) \to \scM_I(S(\mu), S(\nu)) \tag{P3} 
            \end{align}
            is a bijection. 
        \item The conditional distributions satisfy
            \begin{align}\label{eq:CLMforCondLaw} \tag{P4}
                \law_{\hat \PP}(Y_T|Y_o)= S(\law_{ \PP}(X_T|X_o )). 
            \end{align}
    \end{enumerate}
\end{theorem}

\begin{proof}
To show \eqref{eq:Involution}, note that since $S(P)=\law_{\hat\PP}(Y)$, we obtain $S(S(P))=\law_{\bar{\PP}}(Z) $ for $Z := 1/Y$ and $\bar\PP := \frac{Y_T}{\E_{\hat \PP}[Y_T]} \hat \PP$.
Since $Z=1/(1/X)=X$ and 
\[ 
    \bar \PP = \frac{1/X_T}{\E_{\hat{\PP}}[1/X_T]}\hat{\PP} = \frac{1/X_T}{\E_\PP[X_T/(X_T\E_\PP[X_T])]} \frac{X_T}{\E_\PP[X_T]}\PP=\PP,
\]
we obtain $S(S(\PP))=\PP$. 

To show \eqref{eq:ProjOk}, let $R:=\max J$. For continuous bounded $f$ on $\D_J$ we have by the tower property
 \begin{align*}  
 \proj^{\D_J}_\# (S(\pi))(f)= \E_{\PP}\Big[ \frac{f\big((1/X_t)_{t\in J}\big)X_T}{\E_\PP[X_T]} \Big]
= \E_{\PP}\Big[ \frac{f\big((1/X_t)_{t\in J}\big)X_R}{\E_\PP[X_R]} \Big]=(S(\proj^{\D_J}_\# \pi))(f).
\end{align*}

Observe that combining \eqref{eq:Involution} and \eqref{eq:ProjOk} shows \eqref{eq:PlanBij}. Therefore it remains to prove \eqref{eq:CLMforCondLaw}. 
Note that given a continuous bounded function $f$ on $\R_{++}$ we obtain
\begin{align*}
    \law_{\hat \PP}(Y_T|Y_o)(f) &= \E_{\hat\PP}[f(Y_T)|Y_o]  
     = \frac{\E_{\PP} [X_Tf(1/X_T)|X_o]} {\E_\PP[X_T|X_o]} \\
    &= \frac{ \int x_T f(1/x_T) \, d\law_\PP(X_T|X_o)(x_T)}{\int x_T  \, d\law_\PP(X_T|X_o)(x_T)} = S(\law_\PP(X_T|X_o))(f),
\end{align*}
where we have used Bayes' rule and \eqref{eq:def.CLM-pi_unified}.
\end{proof}

\begin{remark}
In our opinion \eqref{eq:PlanBij} is a remarkable observation of \cite{CaDeMe23}. In particular it directly suggests the connection to the martingale transport problem.
\end{remark}
\begin{remark}\label{rmk:CLM_on_kernels}
    We can reformulate \eqref{eq:CLMforCondLaw} in terms of kernels which will be useful later on. 
    Consider $\pi\in\M_{\{0,1\}}(\mu,\nu)$ and let $\law_\PP((X_0,X_1))=\pi$, then we directly obtain from \eqref{eq:PlanBij} and \eqref{eq:CLMforCondLaw} that
    \begin{align}
        \label{eq:S(pi).disintegration}
        S(\pi)(dy_0,dy_1) = S(\mu)(dy_0) S(\pi_{1/y_0})(dy_1).
    \end{align}
\end{remark}

\begin{remark}
    In martingale transport, the pair $(\mu, \nu)$ is called \emph{irreducible} if for all measurable $A,B\subseteq \R, \mu(A), \nu(B)>0$ there exists $\pi \in \M_{\{0,1\}}(\mu, \nu)$ such that $\pi(A\times B)>0$. This condition often allows for cleaner results e.g.\ in the context of dual attainment \cite{BeNuTo16} or the description of optimizers \cite{BaBeHuKa20}. It is a natural assumption in the mathematical finance context, where it corresponds to call prices being strictly increasing w.r.t.\ time to maturity.
    From Theorem \ref{thm:first_properties_CLM} \eqref{eq:PlanBij} it follows that $(\mu, \nu) $ is irreducible if and only if $(S(\mu), S(\nu))$ is irreducible. 
    
    In the present one-dimensional case it is straightforward to see that $(\mu, \nu)  $ can be decomposed into countably many irreducible components (cf.\  \cite[Theorem A.4]{BeJu16}) and, again by Theorem \ref{thm:first_properties_CLM} \eqref{eq:PlanBij} it is plain that the irreducible components of $(\mu, \nu) $ correspond to the ones of $(S(\mu), S(\nu))$. 
\end{remark}

The next theorem covers the relation of the CN transformation and martingale optimal transport problems, as well as weak martingale optimal transport problems. For this reason, we define for measurable $c: \D_I\to [0,\infty]$ the function $S^*(c)$ by
\begin{align}
\label{eq:Costtrans}
S^* (c)(x):= x_T c(1/x), \quad x\in \D_I.
\end{align}
Furthermore, for measurable $C:\cP_1(\R_{++})\to[0,\infty]$, set $S^\ast(C)(\rho) :=  b(\rho) C( S(\rho) )$, $\rho \in \cP_1(\RR)$.

\begin{theorem}\label{thm:second_properties_clm}
    For the martingale optimal transport problem we have
    \begin{align}\label{eq:EquivalentMOT} \tag{P5}
        \frac{1}{b(\mu)}\inf\Big\{ \int S^*(c) \, d\pi: \pi \in \M_I(\mu, \nu)\Big\}=\inf\Big\{ \int c \, d\hat\pi: \hat\pi \in \M_I(S(\mu),S(\nu))\Big\},
    \end{align}
    and $\pi \in \scM_I(\mu,\nu)$ minimizes the left-hand side in \eqref{eq:EquivalentMOT} if and only if $S(\pi)$ minimizes the right-hand side.
    Similarly, for the weak martingale optimal transport problem we have    \begin{multline}\label{eq:EquivalentWMOT} \tag{P6}
        \frac{1}{b(\mu)}\inf\Big\{ \int S^*(C)(\pi_x) \, \mu(dx): \pi \in \M_{\{0,1\}}(\mu, \nu)\Big\} \\ = \inf\Big\{ \int C(\hat \pi_y) \, S(\mu)(dy): \hat\pi \in \M_{\{0,1\}}(S(\mu), S(\nu))\Big\},
    \end{multline}
    and $\pi \in \scM_{\{ 0,1\}}(\mu,\nu)$ minimizes the left-hand side in \eqref{eq:EquivalentWMOT} if and only if $S(\pi)$ minimizes the right-hand side. 
\end{theorem}

\begin{proof}
    To show \eqref{eq:EquivalentMOT}, note that by Bayes' rule we have
    \begin{align*}
        \E_{\hat \PP}[c(Y)]= \frac{\E_{\PP}[X_T c(1/X)]}{\E_\PP[X_T]} =\frac{1}{\E_\PP[X_T]} \E_{\PP } [S^*(c)(X)].
    \end{align*}
    Similarly, to obtain \eqref{eq:EquivalentWMOT} we use \eqref{eq:CLMforCondLaw} to obtain
    \[
        \E_{\hat \PP} [C(\law_{\hat \PP}(Y_1|Y_0))] = 
        \E_\PP \Big[  \frac{X_1}{\E_\PP[X_1]} C(S(\law_{\PP}(X_1|X_0))) \Big] = \frac{1} {\E_\PP[X_1]} \E_\PP[S^*(C)(\law_{\PP}(X_1|X_0))]. \qedhere
    \]
\end{proof}

We remark that most of the above properties \eqref{eq:Involution}-\eqref{eq:EquivalentWMOT} were already established in \cite[Lemma 2.3]{CaLaMa14} while the novelty lies in also covering the weak martingale optimal transport setting.

\section{Geometric stretched Brownian motion}\label{sec:AppgSBM}

\subsection{\gSBM{} as CN-transformation of \SBM}

If $\mu, \nu$ have finite second moment, then the optimization problem 
\begin{equation}\label{eq:SBMprimal}
    V^\SBM(\mu,\nu) := \sup \bigg\{ \E\Big[ \int_0^1 \sigma_t \, dt \Big] : dX_t = \sigma_t \, dB_t, X_0 \sim \mu, X_1 \sim \nu \bigg\}
\end{equation}
has a unique-in-law solution $X^{\SBM}=X^{\SBM, \mu, \nu}$. $X^{\SBM, \mu, \nu}$ or, depending on the context, its law $\QQ^{\SBM, \mu, \nu}$ is called \emph{stretched Brownian motion} 
from $\mu$ to $\nu$. Stretched Brownian motion is a strong-Markov martingale (\cite[Corollary 2.5]{BaBeHuKa20}) which admits a precise structural description in the spirit of Brenier's theorem, see Section \ref{sec:AppgSBMNumerics} below. 

In the investigation of \SBM{} it is key that $V^\SBM(\mu,\nu)$ is also the solution of the wMOT problem 
\begin{align}\label{eq:SBMprimalweak}
      V^\SBM(\mu,\nu)
      =
\sup_{\pi\in \scM_{\{0,1\}}(\mu, \nu)}\int \MCov(\pi_{x},\gamma_1) \, \mu(dx),
\end{align} 
where $\MCov(p_{1},p_{2}) := \sup_{q \in \Pi(p_{1},p_{2})} \int x_1 x_2 \, dq$ and $\gamma_1$ is the centered Gaussian with variance 1. Moreover, 
\eqref{eq:SBMprimalweak} has a unique solution $\pi^{\SBM, \mu, \nu}$ and $(X_t^{\SBM, \mu, \nu})_{t=0,1}\sim \pi^{\SBM, \mu, \nu} $.
On the other hand, it is also straightforward to reconstruct $X^{\SBM, \mu, \nu} $ from $\pi^{\SBM, \mu, \nu}$, see \cite[Theorem 2.2]{BaBeHuKa20}.

Problem \eqref{eq:SBMprimal} can be interpreted as maximizing the correlation with Brownian motion subject to marginal constraints. From a mathematical finance perspective it appears equally natural to maximize the correlation of the \emph{log return} with Brownian motion, e.g.\ to consider geometric stretched Brownian motions which solve 
\begin{equation}\label{eq:gSBMprimal}
    V^\gSBM(\mu,\nu) := \sup \bigg\{ \E\Big[ \int_0^1 \sigma_t \, dt \Big] : dY_t = \sigma_t Y_t \, dB_t, Y_0 \sim \mu, Y_1 \sim \nu \bigg\}.
\end{equation}

In the main result of this section we show that \eqref{eq:gSBMprimal} is equivalent to a wMOT \eqref{eq:gSBMprimalweak} which is the CN-transformation of the wMOT-characterization of \SBM{} in \eqref {eq:SBMprimalweak}. Consequently \gSBM s are precisely the (continuous-time) CN-transformations of \SBM s.

\begin{theorem} \label{thm:gSBM}
    Assume that $S(\nu)$ has finite second moment. Then 
    \begin{align}\label{eq:gSBMprimalweak}
        V^\gSBM(\mu,\nu) = \sup_{\pi\in \scM_{\{0,1\}}(\mu, \nu) }\int x \MCov( S(\pi_x) , \gamma_1 ) \, \mu(dx) \quad \left( = b(\mu)V^\SBM(S(\mu),S(\nu))\right).
    \end{align}
    Problem \eqref{eq:gSBMprimal} has a unique-in-law solution $\QQ^{\gSBM, \mu, \nu}$ (`\emph{geometric stretched Brownian motion}') and $\QQ^{\gSBM, \mu, \nu}$ is given by the continuous-time CN-transformation of the stretched Brownian motion from  $S(\mu)$  to $S(\nu)$, i.e.,
    \begin{align}
       \QQ^{\gSBM, \mu, \nu} = S \left(\QQ^{\SBM, S(\mu), S(\nu)}\right). 
    \end{align}

\end{theorem}

\subsection{Structure of \gSBM{} and numerical simulation}\label{sec:AppgSBMNumerics}

To obtain numerical simulations of SBM and, in turn, gSBM one uses that SBM admits a precise structural description.
Specifically, if $B$ denotes Brownian motion started in a probability $\alpha$ and $f: \R\to \R$ is such that $f(B_1)$ has finite second moment, then the martingale 
\[ M^{\alpha}_t:=\E[f(B_1)| B_t],\quad t\in [0,1] \] is called \emph{Bass martingale}.

Assume that $(\bar \mu, \bar \nu)$ is irreducible and that $\bar \nu$ has finite second moment.
Then \cite[Theorem 3.1]{BaBeHuKa20}, asserts that $X$ is the stretched Brownian motion from $\bar \mu $ to $\bar \nu$ if and only if $X$ is a Bass martingale. 
In this case the so called \emph{Bass measure} $\alpha$ is unique up to translation.

In order to simulate paths from a gSBM between $\mu$ and $\nu$ we can, by Theorem \ref{thm:gSBM}, simulate paths from a SBM between $S(\mu)$ and $S(\nu)$ and convert them. Furthermore, we can assume the pair $(\mu,\nu)$ to be irreducible, since otherwise we can simulate paths on each irreducible component.

Let us therefore recall how to find a Bass measure $\alpha$ for the SBM between $S(\mu)$ and $S(\nu)$, as well as how to simulate paths from a SBM. 
This is well described by the following diagram:
\begin{equation} \label{eq:diagram}
\begin{tikzcd}
    Y_0\sim\mu \arrow[r,"Y_t"]\arrow[d, leftrightarrow, "S"] & Y_1\sim\nu \arrow[d, leftrightarrow,"S"] \\
    X_0\sim S(\mu) \arrow[r, "X_t"] & X_1\sim S(\nu)\\
    B_0\sim\alpha \arrow[r, "B_t"]\arrow[u, "T_{\alpha,S(\mu)}"] & B_1\sim\alpha\ast\gamma_1 \arrow[u, "T_{\alpha\ast\gamma_1,S(\nu)}"]
\end{tikzcd}
\end{equation}
Here, we denote by $T_{\theta_1,\theta_2}$ the monotone transport map from $\theta_1$ to $\theta_2$ and by $\gamma_t$ the centered Gaussian with variance $t$.
Given a distribution $\alpha \in \mathcal P(\mathbb R_{++})$ and denoting by $\phi$ the density of the standard Gaussian, the diagram \eqref{eq:diagram} commutes if and only if $T_{\alpha,S(\mu)} = \phi\ast T_{\alpha\ast\gamma_1,S(\nu)}$, i.e.\ $\alpha$ is the corresponding Bass measure.
Hence, it remains to find a suitable $\alpha$ which can be achieved, for example, by using the fixed-point operator $\mathcal{A}:{\rm CDF}\to {\rm CDF}$, c.f.\ \cite{CoHe21,AcMaPa23,JoLoOb23},
\begin{equation}\label{eq:operator_cdf}
    \mathcal{A}F:=F_{S(\mu)}\circ(\phi\ast(Q_{S(\nu)}\circ(\phi\ast F))),
\end{equation}
where we denote by $F_\theta$ the cdf of the measure $\theta$ and by $Q_\eta$ the quantile function of the measure $\eta$.

The operator $\mathcal{A}$ admits a (up to translations) unique fixed point if $S(\mu)\leq_cS(\nu)$ are in convex order, the pair $(S(\mu),S(\nu))$ is irreducible and certain regularity properties are satisfied, c.f.\ \cite[Assumption 3.1]{AcMaPa23} for details. Assume those are satisfied, then the fixed point $F_\alpha$ is the cdf of the starting law $\alpha$ of the Brownian martingale and the above diagram commutes. 
This enables us to propose an algorithm for simulating paths from a gSBM.

\begin{algorithm}
\caption{A pseudocode for sampling from gSBM}\label{algo:pseudocode_gSBM}
\begin{algorithmic}
    \Require $\mu$ and $\nu$, starting cdf $F^{(0)}$
    \State Calculate $S(\mu)$ and $S(\nu)$
    \State Apply $\mathcal{A}$ starting with $F^{(0)}$ until converged leading $\alpha$ as starting cdf of SBM
    \State Calculate OT map $T_1$ from $\alpha\ast\gamma_1$ to $S(\nu)$
    \State Calculate maps for intermediate times $T_t:=T_1\ast\gamma_{1-t},~t\in[0,1]$
    \State Sample Brownian paths $(B_t)_t$
    \State Transform Brownian paths to SBM $(X_t)_t = (T_t(B_t))_t$
    \State Obtain samples of gSBM $(Y_t)_t$ from  samples of SBM $(X_t)_t$ via acceptance-rejection sampling
\end{algorithmic}
\end{algorithm}

\begin{figure}%[H]
   \centering
   \includegraphics[page=1,width=0.95\textwidth]
       {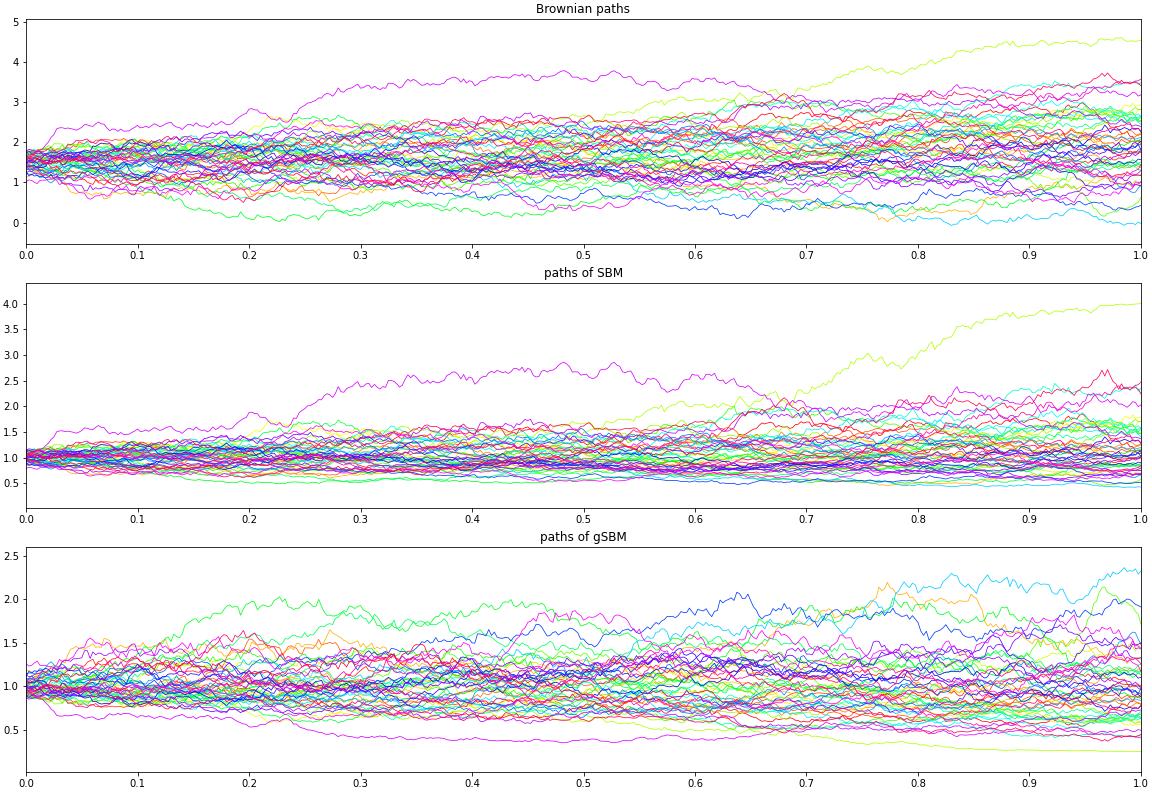} 
 \caption{\vspace{0cm} $50$ paths of BM between $\alpha$ and $\alpha\ast\gamma_1$, of SBM between $S(\mu)$ and $S(\nu)$, as well of gSBM between $\mu$ and $\nu$.}
 \label{fig:simulation}
\end{figure}

\subsection{Proofs}

Before giving the proof of Theorem \ref{thm:gSBM} we need a couple of preparations.
We start by relating the costs in \eqref{eq:gSBMprimal} and \eqref{eq:gSBMprimalweak}.

\begin{lemma}\label{lem:CLM.MCov}
    We have
    \begin{equation} \label{eq:lem.CLM.MCov}
        \sup \bigg\{ \E \Big[ \int_0^1 \sigma_t \, dt \Big] : Y_0 = b(\nu), Y_1 \sim \nu, dY_t = Y_t \sigma_t \, dB_t \bigg\} =
        b(\nu) {\rm MCov}(S(\nu), \gamma_1).
    \end{equation}
\end{lemma}

\begin{proof}
    We will use a classical change of measure argument:
    We fix a stochastic basis $(\Omega,\cF,\mathbb P, (\cF_t)_t)$ supporting a $\mathbb P$-Brownian motion $B$.
    In the first step, consider a positive $\mathbb P$-martingale $X$ starting in $x_0 > 0$ with $dX_t = X_t \sigma_t \, dB_t$ and $\sigma = (\sigma_t)_t$ is $(\cF_t)_t$-adapted.
    Since $x_0 = X_0 = \mathbb E_{\mathbb P}[X_1]$, the measure $\hat{\mathbb P}$ whose density is given by $\frac{d \hat{\mathbb P}}{d \mathbb P} = \frac{X_1}{x_0}$, is a probability measure equivalent to $\mathbb P$.
    By Girsanov's theorem \cite[Theorem 16.19]{Kall02} the process $\hat B = (\hat B_t)_t$ defined by
    \[
        \hat B_t := \int_0^t \frac{1}{X_s} \, d[X,B]_s - B_t = \int_0^t \frac{X_s \sigma_s}{X_s} \, ds - B_t = \int_0^t \sigma_s \, ds - B_t,
    \]
    is a local $\hat{\mathbb P}$-martingale with $[\hat B]_t = t$, thus, a $\hat{\mathbb P}$-Brownian motion.
    As $X$ is strictly positive, we can define $Y = (Y_t)_t$ by $Y_t = 1 / X_t$, from where we deduce by It\^o's formula the dynamics
    \[
        dY_t =  \frac{\sigma_t^2 X_t^2}{X_t^3} \, dt - \frac{\sigma_t X_t }{X_t^2} \, dB_t = \sigma_t Y_t ( \sigma_t dt - dB_t) = \sigma_t Y_t \, d\hat B_t.
    \]
    We observe that $\E_{\hat{\mathbb P}}[Y_1] = \frac{\E_{\mathbb P}[ X_1 Y_1 ]}{x_0} = \frac1{x_0}$ and $\E_{\hat{\mathbb P}}[Y_1 | \cF_t] = \frac{\E_{\mathbb P}[X_1 Y_1 | \cF_t]}{\E_{\mathbb P}[X_1 | \cF_t]} = Y_t$, $\hat{\mathbb P}$-a.s.\ for $t \in [0,1]$.
    In particular, $Y$ is a $\hat{\mathbb P}$-martingale.
    Therefore, we have
    \begin{equation}
        \label{eq:CLM.MCov}
        x_0 \E_{\hat{\mathbb P}}\bigg[ \int_0^1 \sigma_t \, dt \bigg] = \E_{\mathbb P}\bigg[ X_1 \int_0^1 \sigma_t \, dt \bigg] = \E_{\mathbb P} \bigg[ \int_0^1 X_t \sigma_t \, dt \bigg] = \E_{\mathbb P}[ X_1 B_1 ],
    \end{equation}
    where we used first the tower property and that both $X$ and $B$ are $\mathbb P$-martingales.

    Next, assume that additionally $x_0 = b(S(\nu)) = 1 / b(\nu)$ and $X_1 \sim S(\nu)$ under $\mathbb P$.
    By Theorem \ref{thm:first_properties_CLM} \eqref{eq:PlanBij} there is some $\pi \in \scM_{[0,1]}(\delta_{1 / x_0},\nu)$ such that $X \sim S(\pi)$ under $\mathbb P$.
    By construction, we find that $Y \sim S(S(\pi)) = \pi$ under $\hat{\mathbb P}$.

    Due to the preceding observation, we directly derive \eqref{eq:lem.CLM.MCov} from \eqref{eq:CLM.MCov}. 
\end{proof}

\begin{proof}[Proof of Theorem \ref{thm:gSBM}]
First we fix $\pi \in \scM_{\{0,1\}}(\mu,\nu)$.
Due to stability of optimal transport maps, it is well-known that the map $\mathcal P_2(\mathbb R) \ni \eta \mapsto Z^\eta \in L_2(\gamma_1)$ where $Z^\eta \sim \eta$ and $\mathbb E_{B_1 \sim \gamma_1}[ Z^\eta B_1] = {\rm MCov}(\eta,\gamma_1)$, is $(\mathcal W_2, L_2)$-continuous where $\mathcal W_2$ denotes the Wasserstein-2 distance on $\mathcal P_2(\mathbb R)$.
Consider the probability space $(\mathbb R^2, \mathcal B(\mathbb R^2),\mu \otimes \gamma_1)$ with the random variables $X_0(\omega_1,\omega_2) = \omega_1$ and $B_1(\omega_1,\omega_2) = \omega_2$ for $\omega = (\omega_1,\omega_2) \in \mathbb R^2$.
It easily follows that there exists a random variable $X_1 \in L_2(\mu \otimes \gamma_1)$ where a.s.\ $\law(X_1 | X_0) = \pi_{X_0}$ and
\[
    \mathbb E[X_1 B_1 | X_0] = {\rm MCov}(\pi_{X_0},\gamma_1),
\]
whence,
\begin{equation}
    \label{eq:thm.gSBM.0}
    \mathbb E[X_1B_1] = \int {\rm MCov}(\pi_{x_0}, \gamma_1)\, \mu(dx_0).
\end{equation}
Using Lemma \ref{lem:CLM.MCov} and \eqref{eq:thm.gSBM.0}, we compute 
    \begin{align}
        V^{\gSBM}(\mu,\nu) \label{eq:gSBM.static}
        &=
        \sup_{\pi \in \scM_{\{0,1\}}(\mu,\nu)} \int x_0 {\rm MCov}(S(\pi_{x_0}),\gamma_1) \, \mu(dx_0) \\ \nonumber
        &= \sup_{\tilde\pi \in \scM_{\{0,1\}}(S(\mu),S(\nu))} b(\mu)\int {\rm MCov}(\tilde\pi_{y_0}, \gamma_1) \, S(\mu)(dy_0) = b(\mu)V^{\SBM}(S(\mu),S(\nu)),
    \end{align}
    where the second-to-last equality is due to Theorem \ref{thm:first_properties_CLM} \eqref{eq:CLMforCondLaw} and Remark \ref{rmk:CLM_on_kernels}, and identifying $\tilde\pi = S(\pi)$.
    This yields the first assertion.

    Moreover, by the same reasoning it follows that a martingale $X$ (with $dX_t = X_t \sigma_t \, dB_t$ and law $\QQ$) minimizes \eqref{eq:gSBMprimal} if and only if $\pi^\ast := {\rm Law}(X_0,X_1) \in \scM_{\{0,1\}}(\mu,\nu)$ is optimal for the right-hand side in \eqref{eq:gSBM.static} and $\E[\int_0^1\sigma_t \, dt | X_0] = X_0{\rm MCov}(S(\pi_{X_0}),\gamma_1)$ a.s.
    Thus, invoking Theorem \ref{thm:first_properties_CLM}, Theorem \ref{thm:second_properties_clm}, and \cite[Theorem 1.5]{BaBeHuKa20} we find $S(\QQ) = \QQ^{\SBM,S(\mu),S(\nu)}$, which concludes the proof.
\end{proof}

\section{The CN-transformation and adapted Wasserstein distance}

\subsection{(geometric) adapted Wasserstein distance} 
In order to recall the definition of the adapted Wasserstein distance, let us first define the set of bicausal couplings. For this sake, consider $C[0,1]\times C[0,1]$, and the canonical processes $X,X'$, i.e.\ for $(\omega,\omega')\in C[0,1]\times C[0,1]$, consider for all $t\in[0,1]$
\begin{equation*}
    X_t(\omega,\omega')=\omega_t,\quad X'_t(\omega,\omega')=\omega'_t.
\end{equation*}
Given two measures $\QQ_1,\QQ_2\in\mathcal{P}(C[0,1])$ we say that a coupling $\pi\in\Pi(\QQ_1,\QQ_2)$ is causal from $\QQ_1$ to $\QQ_2$, if for $(X,X')\sim \pi$, it holds for all $t\in[0,1]$
\begin{equation}\label{eq:causality_constraint}
    \law((X'_s)_{s\in[0,t]}\mid (X_s)_{s\in[0,1]}) = \law((X'_s)_{s\in[0,t]}\mid (X_s)_{s\in[0,t]}).
\end{equation}
This says that under $\pi$ the conditional law of $X'$ up to some time $t$ given $X$ can only depend on the past of $X$, not on the future. A coupling $\pi\in\Pi(\QQ_1,\QQ_2)$ is then defined to be bicausal between $\QQ_1$ and $\QQ_2$ if it is both causal from $\QQ_1$ to $\QQ_2$, as well as causal from $\QQ_2$ to $\QQ_1$. This means that \eqref{eq:causality_constraint} as well as for all $t\in[0,1]$
\begin{equation}\label{eq:anticausality_constraint}
    \law((X_s)_{s\in[0,t]}\mid (X'_s)_{s\in[0,1]}) = \law((X_s)_{s\in[0,t]}\mid (X'_s)_{s\in[0,t]})
\end{equation}
hold. The set of all bicausal couplings between $\QQ_1$ and $\QQ_2$ is then denoted by $\Pi_{\rm bc}(\QQ_1,\QQ_2)$.
Having recalled the set of bicausal couplings, the adapted Wasserstein distance of martingale measures $\QQ_1, \QQ_2$ is defined as
\begin{align}
    \AW(\QQ_1, \QQ_2):= \inf_{\pi \in \Pi_{\rm bc}(\QQ_1,\QQ_2)} \bigg( \int |x_1 - y_1|^2 \, \pi(dx,dy) \bigg)^\frac12, 
\end{align}
see \cite{BaBaBeEd19a}.

Continuous time adapted Wasserstein distance is used to derive continuity of optimal stopping in \cite{AcBaZa20} and to derive stability properties of pricing, hedging and utility maximization in \cite{BaBaBeEd19a}. Remarkably, it satisfies Talagrand type inequalities with respect to specific entropy, see \cite{Fö22}. There is a rich literature on adapted versions of the Wasserstein distance in discrete time, see e.g.\ \cite{PfPi12, BaBePa21} and the references therein.

\subsection{\gSBM{} as metric projection}\label{sec:AppgSBMMetricProj}

It is known that \SBM{} can be defined as the metric projection of Brownian motion onto $\scM^C_{[0,1]}(\mu,\nu)$, the set of continuous martingales starting in $\mu$ and terminating in $\nu$,

with respect to adapted Wasserstein distance.

Specifically, denoting by $\mathbb W$ the Wiener measure on $C[0,1]$, we have
\begin{proposition}\label{prop:SBMproj}
    The stretched Brownian motion $\QQ^{\SBM,\mu,\nu}$ satisfies
    \begin{align}
        \QQ^{\SBM,\mu,\nu} = \argmin \{ \AW(\QQ, \mathbb W) : \QQ \in \scM^C_{[0,1]}(\mu, \nu)\}.
    \end{align}
\end{proposition}
This is established in \cite[Section 6]{BaBeHuKa20}.

We can give an analogous characterization of \gSBM{} upon defining a suitable geometric adapted Wasserstein distance. 
For a martingale measure $\QQ$ on $C([0,1]; \RR)$ we denote by $L \QQ$ the law of the stochastic logarithm of $X \sim \QQ$. 
We then set for martingale laws $\QQ_1, \QQ_2 $ on $C([0,1]; \RR)$
\begin{align}
    g\AW(\QQ_1, \QQ_2):= \AW(L \QQ_1, L \QQ_2),
\end{align}
which is well-defined with values in $[0,\infty]$. We call $g\AW$ the \emph{geometric adapted Wasserstein distance}.

\begin{example}
    We calculate the geometric adapted Wasserstein distance between the laws of two geometric Brownian motions, i.e.\ for $\sigma, \sigma' \geq 0$ we consider 
    \begin{equation*}
        dX_t=\sigma X_tdB_t,\quad dX'_t=\sigma'X'_tdB'_t,\quad X_0= X_0'=1,
    \end{equation*}
     and call their laws $\QQ_1$ and $\QQ_2$. As $g\AW(\QQ_1,\QQ_2)=\AW(L\QQ_1,L\QQ_2)$ note that for 
    \begin{equation*}
        dZ_t= \sigma dB_t,\quad dZ'_t= \sigma' dB'_t,
    \end{equation*}
    we have $Z\sim L\QQ_1$ and $Z'\sim L\QQ_2$. From this we obtain using \cite[Example 3.4]{BaBaBeEd19a}
    \begin{equation*}
        g\AW(\QQ_1,\QQ_2) = \AW(L\QQ_1,L\QQ_2) = \big\lvert \sigma-\sigma'\big\rvert.
    \end{equation*}
    That is, the geometric adapted Wasserstein distance of geometric Brownian motions is exactly the difference of the volatilities of the respective log-returns.  
\end{example}

In analogy to SBM we define that the geometric stretched Brownian motion is the maximizer of 
\begin{equation}
    \label{eq:def.gSBM.MCov}
    \text{gSBM}_\text{corr}(\mu, \nu):=\max \bigg\{\E\Big[ \int_0^1 \frac{\sigma_t}{X_t} \, dt \Big]:dX_t=\sigma_tdB_t,~X_0\sim\mu,X_1\sim\nu\bigg\}.
\end{equation}
Note that we can also write this as in \eqref{eq:gSBMprimal}, i.e.\
\begin{equation*}
    \text{gSBM}_\text{corr}(\mu,\nu)=\max \bigg\{\E\Big[ \int_0^1 \tilde\sigma_t \, dt \Big]:dX_t=\tilde\sigma_tX_tdB_t,~X_0\sim\mu,X_1\sim\nu\bigg\}.
\end{equation*}
Similarly, we will show below that \eqref{eq:def.gSBM.MCov} is equivalent to the minimization problem
\begin{equation}
    \label{eq:def.gSBM.W2} \text{gSBM}_\text{dist}(\mu, \nu):=
    \min \bigg\{\E \Big[ \int_0^1 \Big|\frac{\sigma_t}{X_t} - 1\Big|^2 \, dt \Big]:dX_t=\sigma_tdB_t,~X_0\sim\mu,X_1\sim\nu\bigg\}.
\end{equation}

\begin{lemma} \label{lem:3.5}
    The optimization problems \eqref{eq:def.gSBM.MCov} and \eqref{eq:def.gSBM.W2} are equivalent in the sense that they have the same optimizers and 
    \[ \text{gSBM}_\text{dist}(\mu, \nu) = 1 - 2 \text{gSBM}_\text{corr}(\mu, \nu) - 2\int \log(x_1) \, \nu(dx_1) + 2\int \log(x_0) \, \mu(dx_0). \]
\end{lemma}

\begin{proof}
    Let $X = (X_t)_t$ be a $\mathbb P$-martingale satisfying the SDE $dX_t = \tilde\sigma_t X_t \, dB_t$ where $B$ is a $\mathbb P$-Brownian motion.
    By \cite[Theorem 23.8]{Kall02} the stochastic exponential is the unique solution to the aforementioned SDE, hence, 
    \[ X_t = X_0 \exp\Big( \int_0^t \tilde\sigma_s \, dB_s - \frac{1}{2}\int_0^t \tilde\sigma_s^2 \, ds \Big). \]
    Consequently, we have
    \begin{equation}
        \label{eq:stochastic exponential}
        \log(X_1) - \log(X_0) = \int_0^1 \tilde\sigma_t \, dB_t - \frac{1}{2}\int_0^1 \tilde\sigma_t^2 \, dt.
    \end{equation}

    Using \eqref{eq:stochastic exponential} we find
    \begin{align*}
        \E\Big[ \int_0^1 (1 - \tilde\sigma_t)^2 \, dt \Big] &= 1 - 2 \E\Big[ \int_0^1 \tilde\sigma_t \, dt \Big] + \E\Big[ \int_0^1 \tilde\sigma_t^2 \,dt \Big]
        \\
        &= 1 - 2 \E\Big[ \int_0^1 \tilde\sigma_t \, dt \Big] + 2\E[ -\log(X_1) + \log(X_0) ].
    \end{align*}
    We observe that the last expected value just depends on the initial and terminal distribution of $X$.
    Consequently, minimizing \eqref{eq:def.gSBM.W2} over all continuous martingales $X$ with $X_0 \sim \mu$ and $X_1 \sim \nu$ is equivalent to maximizing \eqref{eq:def.gSBM.MCov} over the same class.
\end{proof}

Write $\mathbb W_{g}$ for the law of geometric Brownian motion. Then we have
\begin{theorem} \label{thm:logprojection}
    Assume that $S(\nu)$ has finite second moment. Then, $\QQ^{gSBM,\mu,\nu}$ is the unique minimizer of
    \begin{align}
        \min \{g\AW(\QQ, \mathbb W_{g}) : \QQ \in \scM^C_{[0,1]}(\mu, \nu)\}.
    \end{align}
\end{theorem}

\begin{proof}
    Let first $dX_t = \sigma_tX_tdB_t$ with $X_0\sim\mu, X_1\sim\nu$ and $\sigma_t \ge 0$. Furthermore, denote $\mathbb Q=\law(X)$. Note that 
    \begin{equation}\label{eq:sde_stoch_log}
        d\tilde X_t=\sigma_tdB_t,\quad\tilde X_t:=(LX)_t
    \end{equation}
    with $LX$ denoting the stochastic logarithm of $X$ and also $\law(\tilde{X})=L\mathbb{Q}$.

    Let us now show that
    \begin{equation}\label{eq:eq_to_show_Thm42}
        \E\Big[\int_0^1\sigma_tdt\Big] \geq \sup\limits_{\pi\in\cplbc(L\mathbb Q,\mathbb W)}\E_\pi[\tilde X_1W_1].
    \end{equation}

    We obtain this by estimating for a bicausal coupling $\pi\in\cplbc(L\mathbb Q,\mathbb W)$
    \begin{align*}
        \E_\pi[\tilde X_1W_1] = \E_\pi[\langle \tilde X\,, W\rangle_1] &= \E_\pi\Big[\int_0^1 d\langle \tilde X\,,W\rangle_t\Big] =\E_\pi\Big[\int_0^1\sigma_t d\langle B\,,W\rangle_t\Big] \\
        &\leq \E_\pi\Big[\Big(\int_0^1\sigma_td\langle B\rangle_t \int_0^1\sigma_t d\langle W\rangle_t\Big)^{1/2}\Big] = \E_\pi\Big[\int_0^1\sigma_tdt\Big]=\E[\langle \tilde X\,,B\rangle_1],
    \end{align*}
    where the inequality is due to the Kunita-Watanabe inequality. We note that the last term does not depend on the coupling $\pi$ anymore but only on the marginal law, i.e.\ $\law(\tilde X) = L\mathbb Q.$
    This shows \eqref{eq:eq_to_show_Thm42}.

    Recalling the definition of the adapted Wasserstein distance, we obtain
    \begin{equation*}
        \AW(L\mathbb Q,\mathbb W)^2 = \inf_{\pi \in \cplbc(L\mathbb Q, \mathbb W)}\E_{\pi}[\langle \tilde X - W\rangle_1] = \E[\langle \tilde X\rangle_1] + 1 - 2\sup_{\pi \in \cplbc(L\mathbb Q, \mathbb W)}\E_{\pi}[\langle \tilde X\,,W\rangle_1]
    \end{equation*}

    Using \eqref{eq:eq_to_show_Thm42} we get
    \begin{equation}\label{eq:inequality_for_gsbm}
        g\AW(\mathbb Q,\mathbb W_g)^2 = \AW(L\mathbb Q,\mathbb W)^2\geq-2\E\Big[\int_0^1\sigma_tdt\Big] + \E[\langle \tilde X\rangle_1] + 1.
    \end{equation}

    By \cite[Proposition 6.4]{BaBeHuKa20} we have
    \begin{equation}\label{eq:equality-sbm-attainment}
        \inf\limits_{\tilde{\mathbb Q}\in\scM^C_{[0,1]}(S(\mu),S(\nu))}\AW(\tilde{\mathbb Q},\mathbb W)^2 = -\max\limits_{\substack{d\tilde{X}_t=\sigma_tdB_t\\\tilde{X}_0\sim S(\mu),\tilde{X}_1\sim S(\nu)}}2\E\Big[\int_0^1\sigma_tdt\Big] +\E[\langle\tilde{X}\rangle_1]+1,
    \end{equation}
    with attainment, i.e.\ there is an optimizer.
    We can now use the bijection from Theorem \ref{thm:first_properties_CLM}, \eqref{eq:inequality_for_gsbm} and \eqref{eq:eq_to_show_Thm42} to arrive at
    \begin{equation}
        \inf\limits_{\mathbb Q\in \scM^C_{[0,1]}(\mu,\nu)}g\AW(\mathbb Q,\mathbb W_g)^2 = -\max\limits_{\substack{dX_t=\sigma_tX_tdB_t\\ X_0\sim \mu,X_1\sim\nu}}2\E\Big[\int_0^1\sigma_tdt\Big] + \E[\langle \tilde X\rangle_1] + 1.
    \end{equation}
    Since, the geometric stretched Brownian motion is defined as the law of the (unique in law) maximizer of
    \begin{align*}
        \max\bigg\{\E\Big[\int_0^1\sigma_tdt\Big]: dX_t=\sigma_tX_tdB_t,~X_0\sim\mu,~X_1\sim\nu\bigg\},
    \end{align*}
    we get by the preceding observations that it is also the unique minimizer w.r.t.\ the geometric Wasserstein distance against the law of geometric Brownian motion, concluding the proof.
\end{proof}

\section{CN-transformation for continuous-time MOT problems}\label{Sec:AppHT}

In this section we investigate the CN-transformation for continuous martingales and linear cost functions, complementing Theorems \ref{thm:first_properties_CLM} and \ref{thm:second_properties_clm}.
Let $h : [0,1] \times \RR \times \R_+ \to \R$ be a measurable function.
We consider the minimization problem
\begin{equation}
    \label{eq:trevisan huesmann}
    \inf \E \Big[ \int_0^1 h\Big(t,X_t,\frac{\sigma_t}{X_t}\Big) \, dt \Big],
\end{equation}
where the infimum is taken over all continuous martingales $X$ with $X_0\sim \mu, X_1\sim \nu$ and absolutely continuous quadratic variation $d\langle X\rangle_t =: \sigma_t^2 \, dt$.

Recall that $\scM_{[0,1]}^{\rm AC}$ is precisely the set of laws of all martingales with absolutely continuous quadratic variation process. 
By \cite{Ka83} there is a measurable function $\sigma : [0,1] \times C[0,1] \to \R_+$ such that $\int_0^t \sigma_s^2(X) \, ds = \langle X\rangle_t$ $\pi$-almost surely, for every $\pi \in \scM_{[0,1]}^{\rm AC}$.
Then, \eqref{eq:trevisan huesmann} can be equivalently reformulated using laws
\begin{equation}
    \label{eq:trevisan huesmann couplings}
    \mathcal{W}^h(\mu,\nu) := \inf_{\pi \in \scM_{[0,1]}^{\rm AC}(\mu,\nu) } \int_{C[0,1]} c_h(x) \, \pi(dx),
\end{equation}
where $c_h$ is given by
\[
    c_h(x) := \int_0^1 h\Big(t,x_t,\frac{\sigma_t(x)}{x_t}\Big) \, dt.
\]
In order to better understand what happens with \eqref{eq:trevisan huesmann couplings} under change of numeraire, note that by
\eqref{eq:Costtrans} the transformed cost is
\[
    S^\ast(c_h)(x) = x_1 \int_0^1 h\Big(t, \frac{1}{x_t},x_t \sigma_t\Big(\frac{1}{x}\Big)\Big) \, dt.
\]
By It\^o's formula we have that $\frac{\sigma_t(x)}{x_t} = x_t \sigma_t(1 / x)$ $\pi$-almost surely for every $\pi \in \scM_{[0,1]}^{\rm AC}$.
Let us define
\begin{equation}\label{eq:little_s_star}
    s^\ast (h)\Big(t,x_t,\frac{\sigma_t(x)}{x_t}\Big):= x_th\Big(t,\frac{1}{x_t},x_t\sigma_t\Big(\frac{1}{x}\Big)\Big),
\end{equation}
so that we get
\begin{theorem}[CN-Transformation in continuous time] \label{thm:trevisan+huesmann.CLM}    For $\pi \in \scM_{[0,1]}^{\rm AC}(\mu,\nu)$ we have
    \begin{equation}
       \int_{C[0,1]} \int_0^1 s^\ast(h)\Big(t,x_t,\frac{\sigma_t(x)}{x_t}\Big) \, dt \, \pi(dx) =
        b(\mu)\int_{C[0,1]} \int_0^1 h\Big(t,y_t,\frac{\sigma_t(y)}{y_t}\Big) \,dt \, S(\pi)(dy),
    \end{equation}
    given that one of the sides is well-defined.
    In particular, $\pi^\ast \in \scM_{[0,1]}^{\rm AC}(\mu,\nu)$ is optimizer of $\mathcal W^{s^\ast(h)}(\mu,\nu)$ if and only if $S(\pi^\ast)$ is optimizer of $\mathcal W^h(S(\mu),S(\nu))$.
\end{theorem}

\begin{proof}
    We use the same notation as in the proof of Theorems \ref{thm:first_properties_CLM} and \ref{thm:second_properties_clm}.
    Let $X \sim \pi \in \scM_{[0,1]}^{\rm AC}(\mu,\nu)$ and $Y = \frac1X$.
    By It\^o's formula we have (since $\inf\limits_tX_t > 0$, a.s.)
    \[
        dY_t = -\frac{dX_t}{X_t^2} + \frac{d\langle X\rangle_t}{X_t^3}.
    \]
    Thus, the quadratic variation of $Y$ satisfies $d\langle Y\rangle_t = \frac{d\langle X\rangle_t}{X_t^4}$, from where we derive the relation
    \begin{equation}
        \label{eq:quadratic variation}
        \frac{\sigma_t^2(Y)}{Y_t^2} = \frac{d\langle Y\rangle_t}{Y^2_t} = \frac{d\langle X\rangle_t}{X^2_t} = \frac{\sigma_t^2(X)}{X_t^2},
    \end{equation}
    both, $\PP$ and $\hat{\PP}$-almost surely. 
    Due to symmetry reasons, assume w.l.o.g.\ that $(t,x) \mapsto s^\ast(h)(t,x_t,\frac{\sigma_t(x)}{x_t})$ is $dt \otimes \pi$-integrable.
    With the preceding observations at hand, we obtain
    \begin{align*}
        \int_{C[0,1]}\int_0^1 s^\ast(h)\Big(t, x_t, \frac{\sigma_t(x)}{x_t} \Big) \, dt \, \pi(dx) 
        &=
        \E_\PP\Big[ \int_0^1 s^\ast(h) \Big(t,X_t,\frac{\sigma_t(X)}{X_t} \Big) \, dt \Big]
        \\
        &=
         \E_\PP \Big[\int_0^1 X_t h\Big(t,Y_t, \frac{\sigma_t(Y)}{Y_t} \Big) \, dt \Big]
         \\
         &=
         \E_\PP \Big[ X_1 \int_0^1 h\Big( t, Y_t, \frac{\sigma_t(Y)}{Y_t} \Big) \, dt \Big]
         \\
         &=
         \E_{\hat{\PP}} \Big[ \int_0^1 h\Big( t, Y_t, \frac{\sigma_t(Y)}{Y_t} \Big) \, dt \Big]\E_{\mathbb P}[X_1]         \\
         &= b(\mu)\int_{C[0,1]} \int_0^1 h \Big(t, y_t, \frac{\sigma_t(y)}{y_t} \Big) \, dt \, S(\pi)(dy),
    \end{align*}
    where the first and last equality are due to $X \sim \pi$ under $\PP$ and $Y \sim S(\pi)$ under $\hat{\PP}$.
    The second equality follows from \eqref{eq:quadratic variation} and the third is due to $X$ being a martingale under $\PP$.
    This concludes the proof.
\end{proof}

Recall the CN-transformation of the cost $h$ as in \eqref{eq:little_s_star}.
We close this section by highlighting two specific types of cost functions that go well together with the CN-transformation:

\begin{enumerate}[label = (T\arabic*)]
    \item \label{it:type 1 phomo} Assume that $h(t,x,\sigma) = \tilde h(t,x) \sigma$ and $\tilde h \colon [0,1] \times \RR \to \RR$ is Borel and bounded.
    \item \label{it:type funcitonal eq} Assume that $h(t,x,\sigma) = x h(t,1 / x, x^2\sigma)$ and $h$ is bounded from below.
\end{enumerate}
Indeed, for $h$ satisfying \ref{it:type 1 phomo} we have $s^\ast(h)(t,x,\frac{\sigma}{x}) = \tilde h(t,\frac1x) x^2\sigma=h(t,\frac1x,x^2\sigma)$ and note that this type encompasses as a special case ($\tilde h \equiv -1$) the geometric stretched Brownian motion from Section \ref{sec:AppgSBM}.

\begin{corollary}
    Assume that $h$ satisfies \ref{it:type 1 phomo}.
    Then
    \begin{multline}
        \label{eq:trevisan huesmann geometric}
        \inf_{\pi \in \scM_{[0,1]}^{\rm AC}(\mu,\nu)} \int_{C[0,1]} \int_0^1 h\Big(t,x_t, \frac{\sigma_t(x)}{x_t} \Big)\, dt \, \pi(dx)\\=
        b(\mu)\inf_{\hat{\pi} \in \scM_{[0,1]}^{\rm AC}(S(\mu),S(\nu))} \int_{C[0,1]} \int_0^1
        h(t,y_t,\sigma_t(y) )\, dt \, \hat{\pi}(dy).
    \end{multline}
    In particular, $\pi^\ast \in \scM_{[0,1]}^{\rm AC}(\mu,\nu)$ minimizes the right-hand side of \eqref{eq:trevisan huesmann geometric} if and only if $S(\pi^\ast)$ minimizes the left-hand side of \eqref{eq:trevisan huesmann geometric}.
\end{corollary}

Similarly, if $h$ satisfies \ref{it:type funcitonal eq} the optimization problems $\mathcal W^{h}$ and $\mathcal W^{s^\ast(h)}$ coincide, and we have $s^\ast(h)(t,x,\frac{\sigma}{x}) = h(t,x,\frac{\sigma}{x})$.
Again, as a corollary of Theorems \ref{thm:first_properties_CLM} and \ref{thm:second_properties_clm} we find

\begin{corollary}
    Assume that $h$ satisifies \ref{it:type funcitonal eq}. Then
    \begin{equation}
        \label{eq:trevisan huesmann functional}
        \mathcal W^h(S(\mu),S(\nu)) = \mathcal W^h(\mu,\nu).
    \end{equation}
    In particular, $\pi^\ast \in \scM^{\rm AC}_{[0,1]}(\mu,\nu)$ is optimal for the left-hand side of \eqref{eq:trevisan huesmann functional} if and only if $S(\pi^\ast)$ is optimal for the right-hand side of \eqref{eq:trevisan huesmann functional}.
\end{corollary}

\section{Shadow couplings}\label{sec:AppShadow}

Shadow couplings were introduced in \cite{BeJu21} and represent a class of martingale couplings between given marginals $\mu, \nu$. These are parametrized through a so-called \emph{source} $\hat \mu \in \cP_1(\RR \times [0,1])$ whose first marginal is $\mu$ and whose second marginal has no atoms.
We write
\[
    \hat{\scM}_{\{0\}} := \mathcal P_1(\RR \times [0,1]) \quad
    \text{and} \quad
    \hat{\scM}_{\{0,1\}} := \Big\{ \pi \in \mathcal P_1(\R^2_{++} \times [0,1]) : b(\pi_{(x_0,u)}) = x_0,\; \pi\text{-a.s.} \Big\}.
\]

The source measures serve as a parametrization of the set of shadow couplings, which we will now introduce.
We define the set of lifted martingale transport plans
\[
    \hat{\scM}_{\{0,1\}} (\hat \mu, \nu) := \Big\{ \pi \in \hat{\scM}_{\{0,1\}} : ((x,u) \mapsto (x_0,u))_\# \pi = \hat \mu, ((x,u) \mapsto x_1)_\# \pi = \nu \Big\}.
\]
Associated to every source $\hat \mu$, we can define a weak martingale transport problem
\begin{equation}
    \label{eq:shadow_cpl}
    V^{\hat \mu} := \inf_{\pi \in \scM_{\{0,1\}}(\mu,\nu)}
    \int C^{\hat \mu}( \pi_{x_0} ) \, \mu(dx_0),
\end{equation}
where the cost function is given by
\begin{align}\label{eq:ShadowInnerCost}
        C^{\hat \mu}(\eta) := \inf_{ \alpha \in \hat{\scM}_{\{0,1\}}( \delta_m \otimes \hat \mu_m, \eta ) }
    \int (1 - u) \sqrt{1 + x_1^2} \,  \alpha(dx,du) \quad\text{and}\quad m:=b(\eta).
\end{align}

Then the \emph{shadow coupling} $\pi^{\hat \mu}$ between $\mu$ and $\nu$ with source $\hat \mu$
is the unique optimizer to \eqref{eq:shadow_cpl}, cf.\ \cite[Proposition 6.2]{BeJu21}.
\begin{remark}
    We collect the following facts on shadow couplings:
    \begin{enumerate}
        \item The function $(1-u)\sqrt{1+y^2}$ could be replaced by an arbitrary (sufficiently integrable) function $\phi(u) \psi(y)$, where $\phi$ is strictly decreasing and $\psi$ is strictly convex, without changing the definition of the shadow $\pi^{\hat \mu}$.
        \item Let $R:\R_{++}\times [0,1]\to\R_{++}\times [0,1]:  (x,u)\mapsto(x, r(u))$, where $r:[0,1]\to [0,1]$ is increasing and such that $R_\# (\hat \mu)$ has no atom. Then $\pi^{\hat \mu} = \pi^{R_\# (\hat \mu)}$. 

        In particular we would not lose generality by assuming that the second marginal of $\hat \mu$ is the Lebesgue measure. However the present choice is more convenient for the subsequent arguments. 

        \item If $\hat \mu$ is the monotone coupling of its marginals, $\pi^{\hat\mu}$ is the \emph{left-monotone coupling} which is also studied in \cite{BeJu16, HeTo13, BeHeTo15, HoNo19a}. Likewise if $\hat \mu$ is the comonotone coupling of its marginals, $\pi^{\hat\mu}$ is the \emph{right-monotone coupling} (the mirrored version of the left-monotone coupling).

        If $\hat \mu$ is the product coupling of its marginals, then $\pi^{\hat\mu}$ is the so called \emph{sunset coupling}.
    \end{enumerate}
\end{remark}

\begin{figure}
   \centering
   \includegraphics[page=1,trim={0 0 0 5cm},width=0.75\textwidth]
       {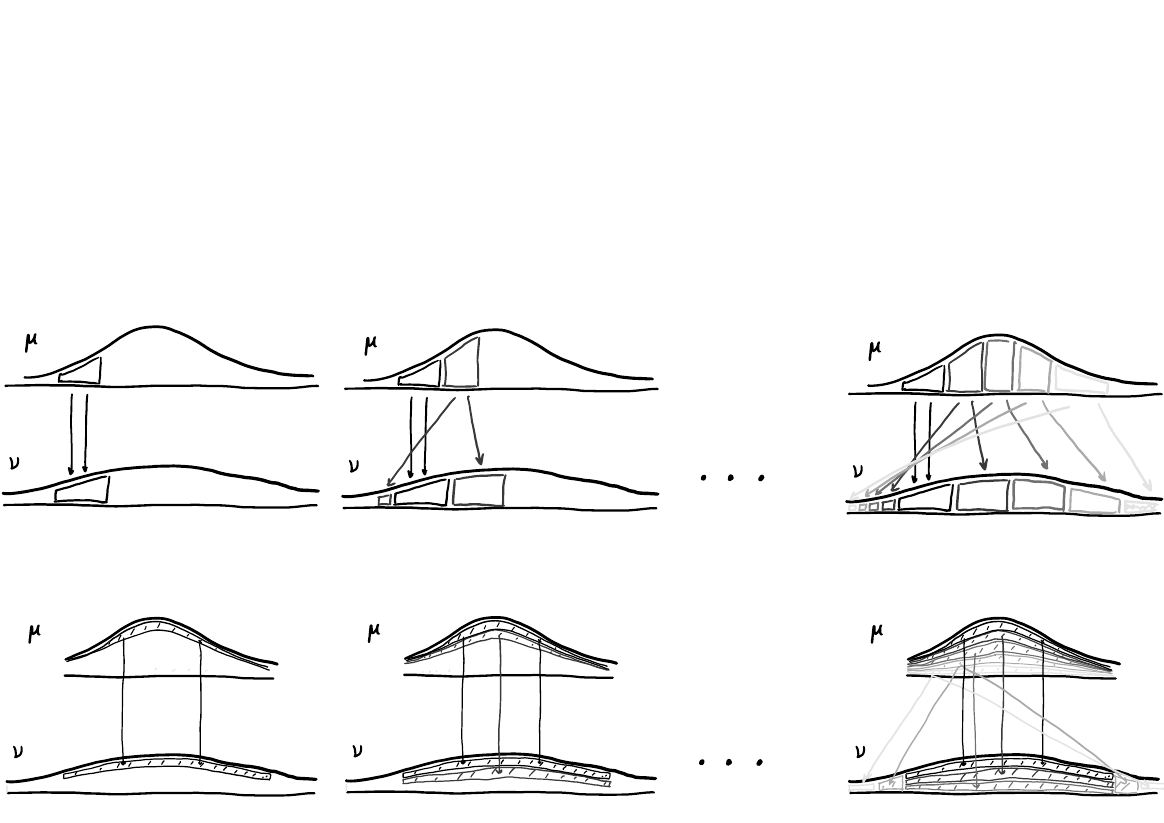} 
 \caption{\vspace{0cm} Depiction of the left-montone martingale coupling and the sunset coupling, see \cite{BeJu21}.}
 \label{fig:Test}
\end{figure}

In order to also cover this extended martingale transport setting, we define a transformation analogous to Section \ref{sec:CLM_trafo_new_sec}.
Let $I = \{0,1\}$ or $I = \{0\}$, then we define \emph{extended} CN-transformation
\begin{equation}
    \label{eq:def hat S}
    \hat S(\pi) : {\hat{\scM}}_I \to {\hat{\scM}}_I, \quad
    \hat S(\pi)(B) := \frac{\int x_T \delta_{(1 / x, u )}(B) \, \pi(dx,du)}{\int x_T \, \pi(dx)},
\end{equation}
for Borel $B \subseteq \R_{++}^{\lvert I\rvert} \times [0,1]$.

In probabilistic terms $ {\hat{\scM}}_{\{0,1\}} $ consists of all $\pi= \law_{\P}(X_0,U,X_1)$ where $U$ has a continuous distribution and $\E_\P[X_1|X_0,U] = X_0$. Writing $\hat \P = \frac {X_1}{\E X_1} \P$ we then have \begin{align}\label{eq:ExtendedS}
    \hat S(\pi) \law_{\hat \P} = \law_{\hat P} (1/X_0, U, 1/X_1).
\end{align}

This allows us to state our main result concerning shadow couplings that is

\begin{theorem}\label{thm:ShadowEqual}
    Let $\pi^{\hat \mu} \in \scM_{\{0,1\}}(\mu,\nu)$ be the shadow coupling with source $\hat \mu$.
    Then $S(\pi^{\hat \mu})\in \scM_{\{0,1\}}(S(\mu), S(\nu))$ is the shadow coupling with source $\hat S(\hat \mu)$.    
\end{theorem}
As a particular consequence of Theorem \ref{thm:ShadowEqual} (and the definition of $\hat S(\hat \mu)$) we obtain that the CN-transformation of the left-monotone martingale coupling is the right-monotone martingale coupling and vice versa. The CN-transformation of the sunset coupling is again the sunset coupling. 

\medskip 

We need some preparations before giving the proof of Theorem \ref{thm:ShadowEqual}.
The extended CN-transformation admits the following counterpart to Theorem \ref{thm:first_properties_CLM}.

\begin{proposition} \label{prop:extended sym op}
    Let $I = \{0 \}$ or $\{0,1\}$.
    The extended CN-transformation satisfies:
    \begin{enumerate}[label = (\roman*)]
        \item $\hat S$ is an involution, i.e.
        \begin{equation}
            \label{hat P1} \tag{\^P1}
            \hat S(\hat S(\pi)) = \pi \quad \text{for }\pi \in \hat{\scM}_I.
        \end{equation}
        \item The map 
        \begin{equation}
            \label{hat P2} \tag{\^P2}
            \hat{S} : \hat{\scM}_{\{0,1\}}(\hat \mu,\nu) \to \hat{\scM}_{\{0,1\}}(\hat S(\hat \mu), S(\nu))
        \end{equation}
        is a bijection whenever $\hat \mu \in \mathcal P_1(\RR \times [0,1])$ and $\nu \in \mathcal P_1(\RR)$.
        \item For every $\pi \in \hat{\scM}_{\{0,1\}}$ with $((x,u) \mapsto (x_0,u))_\# \pi = \hat \mu$, the disintegrations satisfy
        \begin{equation}
            \label{hat P3} \tag{\^P3}
            \hat S(\pi)_{(y_0,u)} = S(\pi_{(x_0,u)})\quad\text{and}\quad
            \hat S(\hat \mu)_{y_0} = \hat \mu_{x_0},
        \end{equation}
        for $x_0 = 1 / y_0$ and $S(\pi)$-a.e.\ $(y,u)$.
    \end{enumerate}
\end{proposition}

\begin{proof} Using \eqref{eq:ExtendedS},
    the proof is completely analogous to Theorem \ref{thm:first_properties_CLM}.  \qedhere
\end{proof}

\begin{lemma} \label{lem:shadow_cpl_cost}
    Let $\hat \mu \in \mathcal P_1(\RR \times [0,1])$ with first marginal $\mu$.
    Then we have for $\mu$-a.e.\ $x_0$ 
    \begin{equation}
        S^\ast(C^{\hat \mu})(\eta) =
        C^{\hat S(\hat \mu)}(\eta)\quad \text{for every }\eta \in \mathcal P_1(\RR)\text{ with }b(\eta) = x_0.
    \end{equation}
\end{lemma}

\begin{proof}
Observe that $f(x) := \sqrt{1 + x^2}$ satisfies the functional equation $f(x) = x f(1 / x)$.
Let us next set $m :=b(\eta)$. This means that for $\alpha \in \hat{\scM}_{\{0,1\}}( \delta_{m} \otimes \hat S(\hat \mu)_{m}, \eta )$ we have 
\begin{equation}
    \label{eq:sqrt identity2}
    m \int (1 - u) f(y_1) \, \hat{S}(\alpha)(dy, du) = 
    \int (1 - u) x_1 f( 1 / x_1) \, \alpha (dx,du) = 
    \int (1 - u) f(x_1) \, \alpha(dx,du).
\end{equation}
Further, it follows from \eqref{hat P2} and \eqref{hat P3} that for $\mu$-a.e.\ $x_0 = 1 / y_0$, $\delta_{y_0} \otimes \hat \mu_{x_0} = \delta_{y_0} \otimes \hat S(\hat \mu)_{y_0}$ and
\begin{equation}
    \label{eq:hatscM in prop}
      \hat S : \hat{\scM}_{\{0,1\}}(\delta_{x_0} \otimes \hat \mu_{x_0}, \eta) \to \hat{\scM}_{\{0,1\}}(\delta_{y_0} \otimes \hat S(\hat \mu)_{y_0},S(\eta)),
\end{equation}
is bijective.
Combining these two observations yields (for $\mu$-a.e.\ $x_0 = 1 / y_0$ and every $\eta \in \mathcal P_1(\RR)$ with $b(\eta) = x_0$)
\begin{align*}
    S^\ast(C^{\hat \mu})(S(\eta)) &=
    x_0 C^{\hat \mu}( \eta ) \\
    &=
    x_0 \inf_{\alpha\in\hat{\scM}_{\{0,1\}}(\delta_{x_0} \otimes \hat \mu_{x_0}, \eta)} \int (1-u)f(x_1)\, \alpha(dx,du)\\
    &=
    x_0 \inf_{\alpha \in \hat{\scM}_{\{0,1\}}(\delta_{y_0} \otimes \hat S(\hat \mu)_{y_0}, S(\eta) )}
    \int (1-u) f(y_1) \, \hat{S}(\alpha) (dy,du)\\
    &=
    \inf_{\alpha \in \hat{\scM}_{\{0,1\}}(\delta_{y_0} \otimes \hat S(\hat \mu)_{y_0}, S(\eta) )} 
    \int (1 - u) f(x_1) \, \alpha(dx,du) = C^{\hat S(\hat \mu)}(S(\eta)),
\end{align*}
where the first equality is due to how $S^\ast$ acts on weak transport cost functions, the third stems from \eqref{eq:hatscM in prop}, and the fourth follows from \eqref{eq:sqrt identity2}.
\end{proof}

\begin{proof}[Proof of Theorem \ref{thm:ShadowEqual}]
    By Theorem \ref{thm:second_properties_clm}, when $\pi^{\hat \mu}$ is minimizer of \eqref{eq:shadow_cpl} then $S(\pi^{\hat \mu})$ is minimizer of
    \[
        \inf_{\pi \in \scM_{\{0,1\}}(S(\mu),S(\nu))} \int S^\ast(C^{\hat \mu})(\pi_{x_0}) \, \mu(dx_0),
    \]
    whence, by Lemma \ref{lem:shadow_cpl_cost} this problem is equivalent to solving \eqref{eq:shadow_cpl} between $S(\mu)$ and $S(\nu)$ with source $\hat S(\hat \mu)$.
    This shows the assertion.
\end{proof}

\bibliographystyle{abbrv}
\bibliography{joint_biblio}

\end{document}